\documentclass[11pt,a4paper]{amsart}
\usepackage{colortbl}

\usepackage{amsthm, amsfonts, amssymb, amsmath, latexsym, enumerate, array}
\usepackage[all,line,poly,rotate,ps,dvips]{xy}
\SelectTips{cm}{}

\usepackage{comment}
\usepackage{amssymb,eucal,graphicx}
\usepackage{enumerate}

\numberwithin{equation}{section}
\newtheorem{theo}{THEOREM}[section]

\newtheorem{cor}[theo]{Corollary}
\newtheorem{prop}[theo]{Proposition}
\newtheorem{dfntn}[theo]{Definition}
\newtheorem{rem}[theo]{Remark}
\newtheorem{claim}[theo]{Claim}

\newtheorem{ex}[theo]{Example}
\newtheorem{fact}[theo]{FACT}
\newtheorem{facts}[theo]{FACTS}
\newenvironment{rem*}{\begin{rem}\em}{\end{rem}}
\newenvironment{ex*}{\begin{ex}\em}{\end{ex}}
\newenvironment{claim*}{\begin{claim}\em}{\end{claim}}
\newenvironment{facts*}{\begin{facts}\em}{\end{facts}}
\newenvironment{fact*}{\begin{fact}\em}{\end{fact}}

\newcommand{\brref}[1]{(\ref{#1})}

\newcommand{\Pin}[1]{{\mathbb P}^{#1}}
\newcommand{\restrict}[2]{{#1}_{\mid _{#2}}}
\newcommand{\lra}{\longrightarrow}

\newcommand{\Projcal}[1]{\mathbb{ P}({\mathcal #1})}

\newcommand{\oofp}[2]{{\mathcal O}_{\mathbb{ P}^{#1}}({#2})}

\newcommand{\scrollcal}[1]{(\Projcal{#1},\tautcal{#1})}
\newcommand {\xel} {(X, L)}

\newcommand{\tautcal}[1]{{\mathcal O}_{\mathbb{P}({\mathcal#1})}(1)}

\newcommand{\num}{\equiv}

\newcommand{\Pp}{\mathbb P}
\newcommand{\Oc}{\mathcal O}

\newcommand{\FF}{\mathbb{F}}
\newcommand{\FFe}{\FF_e}

\usepackage{lineno}

\allowdisplaybreaks[1]

\title[Uniform vector bundles on Hirzebruch surfaces]{On families of rank-$2$ uniform bundles on Hirzebruch surfaces and Hilbert schemes of their scrolls}

\author{Gian Mario Besana}
\address{Gian Mario Besana\\
College of Computing and Digital Media\\ DePaul University\\
1 E. Jackson \\
Chicago IL 60604\\USA}
\email{gbesana@depaul.edu}

\author{Maria Lucia Fania}
\address{Maria Lucia Fania\\ Dipartimento di Ingegneria e Scienze dell'Informazione e Matematica\\
Universit\`{a} degli Studi di L'Aquila\\
Via Vetoio Loc. Coppito\\67100 L'Aquila\\Italy}
\email{fania@univaq.it}

\author{Flaminio Flamini}
\address{Flaminio Flamini\\Dipartimento di Matematica\\ Universit\`a degli Studi di Roma
Tor Vergata \\ Viale della Ricerca Scientifica, 1 - 00133 Roma\\Italy}
\email{flamini@mat.uniroma2.it}



\thanks{Acknowledgments:  The first author acknowledges support from the interim Provosts of DePaul University, Patricia O' Donoghue and David Miller. The second and third author acknowledge partial support by MIUR funds, PRIN project \lq\lq\thinspace Geometria delle
Variet\`a Algebriche''.}

\begin{document}



\begin{abstract} Several families of rank-two vector bundles on Hirzebruch surfaces are shown to consist of all very ample, {\em uniform} bundles. Under suitable numerical assumptions, the projectivization of these bundles, embedded by their tautological line bundles as linear scrolls, are shown to correspond to smooth points of components of their Hilbert scheme, the latter having the expected dimension. If $e=0,1$ the scrolls fill up the entire component of the Hilbert scheme, while for $e=2$ the scrolls exhaust a subvariety of codimension $1. $
\end{abstract}


\maketitle

\section{Introduction}
Vector bundles over smooth, complex, varieties and their moduli spaces, have been intensely studied by several authors over the years (see e.g. the bibliography in \cite{BR} for an overview).

In looking at the landscape of vector bundles over smooth projective varieties, with the eyes of a classical projective geometer, it is natural to wonder about the relationship between that landscape and the parallel world of the families of projective varieties obtained by embedding the projectivized bundles, when possible.  Several authors have investigated Hilbert schemes of projective varieties that arise naturally as embeddings of projectivization of vector bundles, when the appropriate conditions of very ampleness for the bundles themselves, or equivalently for the tautological line bundle on their projectivization, hold.

In the first case of interest, i.e. rank-two, degree d vector bundles over genus g curves $C$, the paper of C. Segre, \cite{Seg}, has to be considered as a corner-stone. Segre's work has indeed inspired several investigations
on surface scrolls in projective spaces (cf.\,e.g.\,\cite{Ghio,APS,GP1,GP3}) as well as a recent systematic study of Hilbert schemes of such surfaces, \cite{CCFMdeg,CCFMnonsp,CCFMsp}.
Morever, the fact that any rank-two vector bundle $\mathcal E$ on $C$ is an extension of line bundles is translated in Segre's language in terms of
Hilbert schemes of unisecants on the ruled surface $\mathbb{P}_C(\mathcal E)$, with fixed degree w.r.t. its tautological line bundle. This viewpoint has been recently considered in \cite{BeFe,CFsp,CFbn,Muk2,Muk},
where the authors study questions on {\em Brill-Noether loci} in the moduli space $U_C(d)$ ($SU_C(L)$, resp.)
of semi-stable, rank-two vector bundles with fixed degree $d$ (fixed determinant $L \in {\rm Pic}^d(C)$, resp.) on $C$, just in terms of extensions of line bundles and Hilbert schemes of unisecants. This series of papers leverages the relationship between the Hilbert scheme approach and the vector-bundle one, in order to obtain results on rank-two vector bundles on curves using primarily projective techniques of embedded varieties.

In attempting to extend this comparative analysis of the two approaches to vector bundles, and correspondingly linear scrolls, over higher dimensional varieties, one is naturally led to consider rank-two vector bundles over ruled surfaces on one hand, and three dimensional linear scrolls embedded with low codimension on the other, as first steps. As far as the latter are concerned, if the codimension is $2,$ it is known \cite{ot1} that there are only four such examples:  the
Segre scroll, the Bordiga scroll, the Palatini scroll, the $K3$-scroll. The first two examples are varieties defined by the maximal minors of an appropriate matrix of linear forms and their Hilbert scheme has being described by Ellingsrud \cite{e}. More generally  the Hilbert scheme of subvarieties
of positive dimension in
projective space which are cut-out by maximal minors
of a matrix with polynomial entries is considered in \cite{fae-fan2,kle}.
As for the Palatini scroll, its Hilbert scheme is described in \cite{fa-me}, while its natural generalizations as Hilbert schemes of scrolls that arise as degeneracy loci of general morphisms
$\phi: {\mathcal O}_{{\Pp}^{2k-1}}^{\oplus m}\lra {\Omega}_{{\Pp}^{2k-1}}(2)$ are studied in \cite{fae-fan} and \cite{fae-fan2}. Further examples of $K3$-scrolls are presented in \cite{fla}.
Turning our attention to rank-two vector bundles over ruled surfaces, one first has the complete classification given by Brosius \cite{bro}, who introduced a canonical way of representing them as extensions of suitable coherent sheaves as in Segre's approach for curves. An equivalent way of obtaining such rank-two bundles as extensions, naturally compatible with Brosius' model, was introduced by Aprodu and Brinzanescu, \cite{BR,ap-br1}, who studied moduli spaces of such vector bundles, independently of any notion of stability.

Notwithstanding this thorough understanding of such bundles, in order to implement our program of investigation, one has to deal with the significant difficulty of establishing the very ampleness of the vector bundle (or tautological line bundle). Even just for rank-two vector bundles over rational ruled surfaces this is a delicate problem. Alzati and the first author, \cite{al-be}, gave a numerical criterion that enables one, in some cases, to establish the necessary very ampleness. This criterion, joined with a need to complete the study of smooth projective varieties of small degree, \cite{fa-li9}, \cite{fa-li10}, \cite{be-bi}, motivated the authors to investigate Hilbert schemes of threefold scrolls given by a particular family of vector bundles $\mathcal{E},$ of rank $2$ over Hirzebruch surfaces $\FFe.$

In a series of three papers, the authors dealt with a family of rank-two vectors bundles $\mathcal{E}$ with first Chern class $c_1(\mathcal{E}) = 3C_0 + \lambda f,$ where $C_0$ and $f$ are the standard generators of the Picard group of $\FFe.$ In particular, if $e = 0, 1,$  in \cite{be-fa} and \cite{be-fa-fl} the authors show that the irreducible component of the Hilbert scheme containing such scrolls is  generically smooth, of the expected dimension, and that its general point is actually a threefold scroll, and thus the corresponding component of the Hilbert scheme is filled up completely by scrolls. Then in \cite{fa-fla}, the second and third author extended the study of the same family of bundles (and scrolls) to the cases with  $ e \geq 2.$ In particular, they showed that, similarly to the previous cases, there exists an irreducible component of the Hilbert scheme containing such scrolls, which is generically smooth, of the expected dimension, with the given scroll corresponding to a smooth point. In contrast to the previous cases though, the family of constructed scrolls surprisingly does not fill up the whole component. A candidate variety to represent the general point of the component was also constructed, and it was shown that one can then flatly degenerate a given scroll to the new variety, in such a way that the base–scheme of the flat, embedded degeneration is entirely
contained in the given component.

In this note we observe that, if one allows the degree of the embedded scrolls to be relatively high, the criterion in \cite{al-be} establishes the very ampleness of other families of vector bundles $\mathcal{E}$ over $\FFe, $ with $c_1(\mathcal{E})= 4C_0 + \lambda f.$ All very ample rank-two bundles in these families are shown to be uniform, with splitting type $(3,1),$ see Proposition \ref{Euniform}. Investigating fully the  cohomological properties of the scrolls considered here would require an extensive enumeration of possible cases, according to sets of values for the parameters involved, see Remark \ref{rem:osserva}, that goes beyond the scope of this note. Therefore, in the second part of this work, scrolls over surfaces $\FFe$ with $e\le 2$ are considered, for which convenient cohomology vanishing can be obtained, see Theorem \ref{thm:Hilbertscheme}.
Nonetheless, the overall framework is quite general and could be adapted to encompass the rest of the bundles in the identified uniform families. Under the new assumptions, Theorem \ref{thm:Hilbertscheme} shows that the scrolls under consideration are smooth points of a component of the Hilbert scheme of embedded projective varieties with the same Hilbert polynomial, with the expected dimension. Leveraging both the vector bundle approach and the Hilbert scheme one, Theorem \ref{thm:parcount} shows that scrolls obtained from our families of vector bundles fill up their component of the Hilbert scheme if $e=0,1$ but exhaust a subvariety of codimension $1$ when $e=2.$ The last result extends to this new class of scrolls results from \cite{fa-fla}.
Beyond the obvious goal of extending these results to scrolls over $\FFe$ for $e\ge3,$ a few other natural questions arise. Following \cite{fa-fla}, in the cases in which our scrolls fill out a positive codimension subvariety of the component of the Hilbert scheme, can one describe a variety $Z$ which is a candidate to represent the general point of such component?
Assuming that a description of a variety $Z$ as above is achieved, can one interpret the projective degeneration
of $Z$ to a scroll of ours in terms of vector bundles on Hirzebruch surfaces?
All these questions will be addressed in forthcoming works.

It is a pleasure for the authors to dedicate this note to Emilia Mezzetti, in the occasion of her recent life milestone.


\section{Notation and Preliminaries}
\label{notation}
Throughout this paper we will use the following notation:
  \begin{enumerate}
\item [ ] $X$ is a smooth, irreducible, complex projective variety of dimension $3$ (or simply a $3$-fold);
\item[] $\Oc_X$ is the structure sheaf of $X;$
\item[ ]$\chi(\mathcal F)=\sum_{i=0}^3 (-)^ih^i(X,\mathcal{F})$ is the Euler characteristic of any coherent sheaf $\mathcal F$  on $X$;
\item[ ]  $\restrict{\mathcal F}{Y}$ is the restriction of $\mathcal F$ to any subvariety $Y \subset X$;
\item[ ]  $K_X$ (or simply $K$, when the context is clear) is a canonical divisor on $X;$
\item[ ] $c_i = c_i(X)$ is the $i^{th}$ Chern class
of $X;$
\item[ ] $c_i(\mathcal{E})$ is the $i^{th}$ Chern class
of a vector bundle $\mathcal{E}$ on $X;$
\item[ ] if $L$ is a very ample line bundle on $X,$ then $d = \deg{X} = L^3$ is the degree of $X$ in the embedding
given by $L;$
\item[] if $S$ is a smooth surface, $\equiv$ will denote the numerical equivalence of divisors on $S$ and, if $W \subset S$ is any closed subscheme, we will simply denote by ${\mathcal J}_{W}$ its ideal sheaf in $\mathcal O_S.$
\end{enumerate}

Cartier divisors, their associated line bundles and the
invertible sheaves of their holomorphic sections are used with no
distinction. Mostly additive notation is used for their group. Juxtaposition is used to denote intersection of divisors.
For any notation and terminology not explicitly listed here, please refer to \cite{H}.

\begin{dfntn} Let $X$ be a $3$-fold and $L$ be an ample line bundle on $X$. The pair $(X, L)$ is called
a {\em scroll} over a normal variety $Y$ if there exist an ample line
bundle $M$ on $Y$ and a surjective morphism  $\varphi: X \to Y$, with
connected fibers, such that $K_X + (4 - \dim Y) L = \varphi^*(M).$
\end{dfntn}

When $Y$ is smooth and $\xel$ is a scroll over $Y$,  then (cf.\,\cite[Prop. 14.1.3]{BESO})
$X \cong \Projcal{E} $, where ${\mathcal E}= \varphi_{*}(L)$ and $L$ is the tautological  line bundle on $\Projcal{E}.$ Moreover, if $S \in |L|$ is a smooth divisor,
then (see e.g. \cite[Thm. 11.1.2]{BESO}) $S$ is the blow up of $Y$ at $c_2(\mathcal{E})$ points; therefore $\chi({\mathcal O}_{Y}) = \chi({\mathcal O}_{S})$ and
\begin{equation}\label{eq:degree}
d : = L^3 = c_1^2(\mathcal{E})-c_2(\mathcal{E}).
\end{equation}

In this paper, we will consider three dimensional scrolls $X$ whose base, $Y ,$  is the Hirzebruch surface $\FF_e= \Pp(\Oc_{\Pp^1} \oplus\Oc_{\Pp^1}(-e))$,
with $e \ge 0$ an integer. If $\pi : \FF_e \to \Pp^1$ denotes the natural projection, then
${\rm Num}(\FF_e) = \mathbb{Z}[C_0] \oplus \mathbb{Z}[f],$ where $C_0$ is the unique section corresponding to
$\Oc_{\Pp^1} \oplus\Oc_{\Pp^1}(-e) \to\!\!\!\to \Oc_{\Pp^1}(-e)$ on $\Pp^1$, and $f = \pi^*(p)$, for any $p \in \Pp^1.$
In particular, it is $C_0^2 = - e, \; f^2 = 0, \; C_0f = 1.$

Let $\mathcal{E}$ be a rank-two vector bundle over $\FF_e.$  Then $c_1( \mathcal{E}) \num a C_0 + c f$, for some $ a, c \in \mathbb Z$, and $c_2(\mathcal{E})=\gamma \in \mathbb Z$. In this context, following Aprodu and Brinzanescu, \cite[\S\,1]{ap-br1},  one can consider two numerical invariants associated to $\mathcal{E},$ as follows :

\begin{itemize}
\item[$(i)$] Let $f \simeq \Pp^1$ be a general fibre of the map $\pi.$ Then  $\mathcal{E}_{|f}\cong \Oc_{f} (d_1)\oplus \Oc_{f}(d_2)$, where the pair $(d_1, d_2)$ is called the {\it generic splitting type} of $\mathcal{E}$, and where $d_2\le d_1,$ $d_1+d_2=a.$ Such an integer  $d_1$ is the first numerical invariant of  $\mathcal{E}.$

\item[$(ii)$] The integer $r$ defined as:
$$-r : = {\rm Inf} \; \{\ell \in \mathbb{Z} \; | \; \, \, H^0(\mathcal{E}(-d_1 \;C_0 + \ell \, f)) = H^0(\mathcal{E}(-d_1 \;C_0) \otimes \pi^*(\mathcal O_{\mathbb{P}^1} (\ell)))\neq 0\}$$ is the second numerical invariant of $\mathcal{E}.$
\end{itemize}

Recall that the vector bundle $\mathcal E$ is said to be {\em uniform} if the splitting type $(d_1, d_2)$ as in (i) is constant for any fibre $f$ (cf.\,e.g.\,\cite[Definition\,3]{ap-br1}).

%
%

\section{A family of  rank-two uniform vector bundles over $\FF_e$.} \label{S:vectorbundles}

Alzati and Besana,\,\cite[p.\,1211,\,Example 1]{al-be}, considered rank-two vector bundles on  $\FF_e$  constructed in the following way:
let
\begin{equation}\label{eq:al-bea}
L \num C_0 + b_{l} f \text{     and    } M \num 3 C_0 +b_{m} f
\end{equation}
be two line bundles on  $\FF_e$  with the assumption
\begin{equation} \label{esistenzafibra}
b_{m}-b_{l}-e-2>0,
\end{equation}
and let $W$ be a zero-dimensional subscheme consisting of two distinct, reduced points on a fixed fibre  $\overline{f}.$
Then one gets a rank-two vector bundle ${\mathcal E}$  fitting in the following exact sequence
\begin{equation}\label{eq:al-be1}
0 \to L \to {\mathcal E} \to M\otimes {\mathcal J}_{W} \to 0.
\end{equation}
Assuming that
\begin{equation}\label{ass x v.a.}
b_l > -2 \text{\ \ and \ \ } b_m\ge 3e+6,
\end{equation}
it follows that the vector bundle ${\mathcal E}$ is very ample (see \cite[Theorem 4.2]{al-be}).

The second of the assumptions \brref{ass x v.a.} can be written as  $b_m= 3e+6+t$ with $t \ge 0.$ Thus, from \eqref{esistenzafibra}, we get
   \begin{equation}\label{boundbl}
   b_{l} < 2e+4+t
   \end{equation}From \eqref{eq:al-be1}, in particular, one has $$c_1({\mathcal E}) = 4C_0 + (b_m+b_l)f \;\; {\rm and} \;\; c_2({\mathcal E}) = L\cdot M+2= \gamma.$$

For simplicity of notation, set  $b_l=b.$  Under our assumptions, we have   $$b_l+b_m=b+3e+6+t \text{\ \  and \ \ } c_2({\mathcal E}) = \gamma= 3b+8+t.$$

\begin{prop} \label{Euniform} Let  ${\mathcal E}$ be  any rank-two vector bundle as in (\ref{eq:al-be1}), for which  assumptions (\ref{ass x v.a.}) and \brref{boundbl} hold. Then  ${\mathcal E}$  is uniform,  of splitting type $(3,1)$.
\end{prop}
\begin{proof}  Since  ${\mathcal E}$ is a very ample,  rank-two vector bundle with $c_1({\mathcal E}) = 4C_0 + (b_m+b_l)f$ then the generic splitting type of ${\mathcal E}$ is either $(3,1)$ or $(2,2)$.

\begin{claim}\label{cl:1}
$(2,2)$ cannot occur as generic splitting type.
\end{claim}

\begin{proof}[Proof of Claim \ref{cl:1}] With notation as above, consider, as in \cite[Theorem 1]{ap-br1}, the following integer
$\ell(c_1, c_2, d_1, r):= \gamma+a(d_1e-r)-(b_l+b_m)d_1+2d_1r-d_1^2e$ and
assume by contradiction that $(2,2)$ occurs as generic splitting type, i.e. $d_1 = d_2 =2.$ Then, with our notation and under our numerical
assumptions, it follows that
\begin{alignat}{1}
 \ell(c_1, c_2, 2, r)&= 3b+8+t+4(2e-r)-2(3e+b+6+t)+4r-4e\\ \notag
  &=b-t-2e-4.\nonumber
 \end{alignat}
 By (\ref{boundbl}) it follows that $ \ell(c_1, c_2, d_1, r)<0$ which contradicts \cite[Theorem 1]{ap-br1}.
Thus $(2,2)$ cannot occur as generic splitting type.
\end{proof}

Claim \ref{cl:1} implies that ${\mathcal E}$ has generic splitting type $(3,1).$
To show that $\mathcal{E}$ is uniform, by \cite[Corollary 5]{ap-br1}, it is enough  to show that $ \ell(c_1, c_2, 3, r)=0.$ In order to compute $\ell,$  the invariant $r$ must first be considered. Tensoring the exact sequence (\ref{eq:al-be1}) by $-3C_0 \otimes \pi^*(  \Oc_{\Pp^1}(\ell)) = - 3 C_0 + \ell f$ gives
\begin{equation}\label{eq:al-betwisted}
0 \to -2C_0+(b+\ell)f \to {\mathcal E}(-3C_0 + \ell f) \to (3e+6+t+\ell)f\otimes {\mathcal J}_{W} \to 0.
\end{equation}
Note that $H^0( -2C_0+(b+\ell)f)=0$ and, by Serre duality,  $H^1( -2C_0+(b+\ell)f)=H^1( -(e+2+b+\ell)f)=H^1(\Pp^1, \Oc_{\Pp^1}(-(e+2+b+\ell)))=0$  if $-e-2-b-\ell\ge -1,$ that is if
\begin{equation}\label{bound1}
\ell \le -e-1-b.
\end{equation}
In this range it follows that $H^0({\mathcal E}(-3C_0 + \ell f))=H^0((3e+6+t+\ell)f\otimes {\mathcal J}_{W})\neq 0$ if and only if $3e + 6 + t + \ell \ge 1$ as, by construction, W is a zero-dimensional scheme consisting of two distinct points on a fibre. Thus
\begin{equation}\label{bound2}
\ell \ge -3e-5-t.
\end{equation}
Notice that for  (\ref{bound1}) and (\ref{bound2}) to be compatible it must be $b\le 2e+4+t,$ which certainly holds by \eqref{boundbl}. Thus we conclude that $r=3e+5+t.$

We finally can compute $$\ell(c_1, c_2, 3, r)=3b+8+t+4(3e-3e-5-t)-3(b+3e+6+t)+6(3e+5+t)-9e=0,$$ which implies that
${\mathcal E}$ is uniform.
\end{proof}

 Let ${\mathcal E}$ be a rank-two vector bundle over $\FF_e$ as in Proposition \ref{Euniform}.  Then
\begin{equation}\label{eq:C1C2}
c_1({\mathcal E}) \num  4 C_0 + (b+3e+6+t) f,\;\;\; c_2({\mathcal E}) = \gamma=3b+8+t, \;\; {\rm for} \;\; t \geq 0,
\end{equation}and  ${\mathcal E}$ is uniform, of splitting type $(3,1)$. Thus, (see \cite[Prop.7.2]{al-be} and \cite{bro}),  there exists an exact sequence
\begin{equation}\label{eq:al-beca}
0 \to A \to {\mathcal E} \to B \to 0,
\end{equation}where $A$ and $B$ are line bundles on $\FF_e$ such that
\begin{equation}\label{eq:al-beca3}
A \num 3 C_0 + (3e+5+t) f \;\; {\rm and} \;\; B \num C_0 + (b+1) f.
\end{equation}

From \eqref{eq:al-beca}, in particular, one has $c_1({\mathcal E}) = A + B \;\; {\rm and} \;\; c_2({\mathcal E}) = A\cdot B$. Note also that since ${\mathcal E}$ is very ample it follows that $B$ is ample and thus $b>e-1.$
\vspace{3mm}

Using \eqref{eq:al-beca}, we can compute cohomology of ${\mathcal E}$, $A,$ and $B$. Indeed, we have:

\begin{prop}\label{prop:com}  Let
${\mathcal E}$ be a rank-two vector bundle over $\FF_e$ as in Proposition \ref{Euniform}. Then
$$
h^i({\mathcal E}) = h^i(A) = h^i(B) =0, \;\mbox{for}\;\; i \geq 1  \;\; {\rm and} \;\; h^0({\mathcal E}) = 5e+2b+4t+28. $$
\end{prop}

\begin{proof} For dimension reasons, it is clear that
$h^j({\mathcal E}) =  h^j(A) = h^j(B) = 0, \; j \geq 3$. Recalling that  $K_{\FF_e} \equiv - 2 C_0 - (e+2) f,$ and using Serre duality, it is:
 $$h^2(A) = h^0( - 5C_0 - (4e+7+t)f) = 0 \ \  \text{ and}
\ \  h^2(B) = h^0( - 3 C_0 - (e +3+b)f) = 0.$$
In particular, this implies that  $h^2({\mathcal E}) = 0.$ In order to show that  $h^1(B) = h^1(A) = 0$ first notice that
$R^1\pi_*(B) =R^1\pi_*(\mathcal{O}(1))  \otimes  \oofp{1}{b+1} = 0$ and
$ R^1\pi_*(A) = R^1\pi_*(\mathcal{O}(3))  \otimes  \oofp{1}{3e+5+t} = 0$ (see for example \cite[p.253]{H}). Recalling that $b > e-1,$ and $t \ge 0,$ Leray's isomorphism then gives

\begin{eqnarray*}
h^1(B)  &=& h^1(\Pp^1, R^{0}\pi_{*}( C_0 + (b+1) f))  \\\nonumber
&=& h^1(\Pp^1, (\Oc_{\Pp^1}\oplus \Oc_{\Pp^1}(-e))\otimes \Oc_{\Pp^1}(b+1)) \\\nonumber
&=&h^1(\Pp^1, \Oc_{\Pp^1}(b+1))+h^1(\Pp^1, \Oc_{\Pp^1}(b+1-e))=0,
\end{eqnarray*}
and
\begin{eqnarray*}
h^1(A) &=& h^1(\Pp^1, R^{0}\pi_{*}(3 C_0 + (3e+5+t) f))  \\\nonumber
&=& h^1(\Pp^1, Sym^3(\Oc_{\Pp^1}\oplus \Oc_{\Pp^1}(-e))\otimes \Oc_{\Pp^1}(3e+5+t))
\\\nonumber
&=&h^1(\Pp^1,( \Oc_{\Pp^1}\oplus\Oc_{\Pp^1}(-e)\oplus\Oc_{\Pp^1}(-2e)\oplus\Oc_{\Pp^1}(-3e))\otimes \Oc_{\Pp^1}(3e+5+t)) \\\nonumber
&=&h^1(\Pp^1, \Oc_{\Pp^1}(3e+5+t)\oplus\Oc_{\Pp^1}(2e+5+t)\oplus \Oc_{\Pp^1}(e+5+t)\oplus \Oc_{\Pp^1}(5+t))=0.
\end{eqnarray*}

Similarly we get that
\begin{eqnarray*}
h^0(A)  &=& h^0(\Pp^1, \Oc_{\Pp^1}(3e+5+t)\oplus\Oc_{\Pp^1}(2e+5+t)\oplus \Oc_{\Pp^1}(e+5+t)\oplus \Oc_{\Pp^1}(5+t))\\\nonumber
&=&3e+6+t+2e+6+t+e+6+t+6+t  \\\nonumber
&=&6e+4t+24,\\\nonumber
\\
h^0(B)  &=& h^0(\Pp^1, \Oc_{\Pp^1}(b+1)\oplus\Oc_{\Pp^1}(b+1-e)) \\\notag
&=& 2b+4-e, \\\nonumber
\end{eqnarray*}and thus
$$h^0({\mathcal E})  = h^0(A)+h^0(B) = 6e+4t+24+ 2b+4-e= 5e+2b+4t+28.$$
\end{proof}

%

\section{$3$-dimensional scrolls over $\FF_e$ and their Hilbert schemes} \label{S:scroll3folds}
As all vector bundles in the the family introduced in \S\,\ref{S:vectorbundles} are very ample, they give rise to a corresponding family of threefolds embedded in projective space as linear scrolls over $\FF_e.$ In this section we will show that such threefolds correspond to smooth points of suitable components of the appropriate Hilbert scheme, and that in some cases (e.g.\,$e=2$) they fill up only a codimension $e-1$ subvariety of such a component.

Let  ${\mathcal E}$ be a very ample rank-two vector bundle over $\FF_e$ as in Proposition \ref{Euniform}, with $c_1({\mathcal E})$ and
$c_2({\mathcal E})$ as in \eqref{eq:C1C2}. As observed above, ${\mathcal E}$ fits in an exact sequence
as in \eqref{eq:al-beca}, where $A$ and $B$ are as in \eqref{eq:al-beca3}. With this set up, let $\scrollcal{E}$ be the 3-dimensional scroll over $\FF_e,$
and $\varphi: \Pp({\mathcal E}) \to  \FF_e$ be the usual projection.

\begin{prop}\label{prop:Xscroll}
The tautological line bundle $L = \Oc_{\Pp({\mathcal E})}(1)$ defines an embedding
$$\Phi:= \Phi_{|L|}: \, \Pp({\mathcal E}) \hookrightarrow  X \subset \Pin{n},$$where $X = \Phi( \Pp({\mathcal E}))$ is smooth, non-degenerate, of degree $d$, with
\begin{equation}\label{eq:ndnew}
n = 5e+2b+4t+27  \;\;\; {\rm and} \;\;\; d = L^3 =8e+5b+7t+40.
\end{equation}
Moreover,
\begin{equation}\label{eq:vannew}
h^i(X, L) = 0, \;\; i \geq 1.
\end{equation}
\end{prop}
\begin{proof} The very ampleness of $L$ follows from that of  ${\mathcal E}$. The expression for the degree $d$ of $X$ in \eqref{eq:ndnew} follows from \eqref{eq:degree}.  Leray's isomorphisms and  Proposition \ref{prop:com} give
\eqref{eq:vannew} whereas the first part of \brref{eq:ndnew} follows from Proposition \ref{prop:com}, because $n+1 = h^0(X,L) = h^0(\FF_e, {\mathcal E}).$
\end{proof}

%

\subsection{The component of the Hilbert scheme containing [$X$]}

In what follows, we are interested in studying the Hilbert scheme parametrizing closed subschemes of
$\Pp^{n}$ having the same Hilbert polynomial $P(T):=P_{X}(T) \in \mathbb{Q}[T]$ of $X$, i.e. the numerical polynomial defined by
$$P(m)= \chi(X, mL)= \frac{1}{6}m^3L^3-\frac{1}{4}m^2L^2\cdot K+\frac{1}{12}mL\cdot(K^2+c_2)+\chi(\Oc_{X}),  \mbox{for all $m \in \mathbb{Z}$},$$(cf.\,\cite[Example 15.2.5]{Fu}). For basic facts on Hilbert schemes we refer to e.g. \cite{Se}.

A scroll $X \subset \Pp^{n},$ as above, corresponds to a point $[X] \in \mathcal H_3^{d,n}$, where $\mathcal H_3^{d,n}$ denotes the  Hilbert scheme parametrizing closed subschemes of $\Pp^{n}$ with Hilbert polynomial $P(T)$ as above,
where $n$ and $d$ are as in \eqref{eq:ndnew}. Let
\begin{equation}
N : = N_{X/\Pp^{n}}
\end{equation}
denote the normal bundle of $X$ in ${\mathbb P}^{n}$. From standard facts on Hilbert schemes (see, for example, \cite[Corollary 3.2.7]{Se}), one has
\begin{equation}\label{eq:tangspace}
T_{[X]} (\mathcal H_3^{d,n}) \cong H^0(N)
\end{equation}
and
\begin{equation}\label{eq:expdim}
h^0(N) - h^1(N) \le \dim_{[X]}(\mathcal H_3^{d,n}) \le h^0(N),
\end{equation}
where the left-most integer in \eqref{eq:expdim} is the {\em expected dimension} of $\mathcal H_3^{d,n}$ at $[X]$ and
where equality holds on the right in  \eqref{eq:expdim} if and only if $X$ is {\em unobstructed} in $\Pp^{n}$ (namely, iff
$[X] \in \mathcal H_3^{d,n}$ is a smooth point).

In the next result we exhibit components of the Hilbert schemes of scrolls, as in Proposition \ref{prop:Xscroll},
which are generically smooth and of the expected dimension.

\begin{theo}\label{thm:Hilbertscheme} Assume $e \leq 2$ and $b = 2e + 3 + t$. Then, for any $t \geq 0$,
there exists an irreducible component $\mathcal X_e \subseteq \mathcal H_3^{d,n}$, which is generically smooth and of (the expected) dimension
\begin{equation}\label{eq:expdimXe}
\dim(\mathcal X_e) =  n(n+1) + 9e + 20 + 6t,
\end{equation} where $n = 9 e + 33 + 6t$ and $d = 18 e + 55 + 12 t$, such that $[X]$ sits in the smooth locus of $\mathcal X_e$.
\end{theo}

\begin{rem}
\label{rem:osserva}
\begin{itemize}
\normalfont{
\item[(i)] Notice that, for any fixed integer $t \geq 0$, $b =2e + 3 + t$ is the maximal value that $b$ can achieve according to \eqref{boundbl}. Correspondingly, one can compute $n$ and $d$ by simply substituting this value of $b$ in \eqref{eq:ndnew}, which exactly gives $ n = 9 e + 33 + 6t$ and $d = 18 e + 55 + 12 t$ as in the statement of Theorem \ref{thm:Hilbertscheme}.
\item[(ii)] Assumptions $e \leq 2$ and $b = 2e + 3 + t$ in Theorem \ref{thm:Hilbertscheme} are not inherently imposed by the problem, but added here only to simplify
technical details in the proof of Theorem \ref{thm:Hilbertscheme} (see details of proof below). One could as well consider all cases $-2 < b \leq 2e + 3 + t$, but  then an exhaustive analysis  of all possible numerical values for  $b$, $e$ and $t,$ compatible with \eqref{esistenzafibra}, \eqref{ass x v.a.}, \eqref{boundbl}, would be required, as well as thorough parameter computations for the construction of threefolds $X$ as in Proposition \ref{prop:Xscroll}. This approach would be beyond the scope of this note.
Thus, to simplify the proof of Theorem \ref{thm:Hilbertscheme}, we will assume $b = b_l < 6 + t + e$ which, by \eqref{boundbl},
is indeed the case as soon as $e \leq 2$, as well as  $b \geq 2e + 3 + t$ which, together with \eqref{boundbl}, gives exactly $b= 2e + 3 + t$.  We would like to stress that, even under the conveniently chosen numerical assumptions, Theorem \ref{thm:Hilbertscheme} gives infinitely many classes of examples of Hilbert schemes: for any $e \in \{0,1,2\}$ and for any integer $t \geq 0. $ Indeed one notices that \eqref{esistenzafibra} is always satisfied when $b_m = 3e + 6+t$ and $b_l = b= 2e + 3 + t$, for any $t \geq 0.$
\item[(iii)] Under the numerical assumptions of Theorem \ref{thm:Hilbertscheme}, all bundles $\mathcal E$ will split (cf.\,\eqref{eq:split}).}
\end{itemize}
\end{rem}

\begin{proof}[Proof of Theorem \ref{thm:Hilbertscheme}] By \eqref{eq:tangspace} and \eqref{eq:expdim}, the statement will follow by showing that $H^i(X,N)=0$, for $i \geq 1$, and computing $h^0(X,N) = \chi(X, N)$. The necessary arguments, and cohomological  computations, run as in \cite[Theorem 4.4]{fa-fla}, with appropriate obvious modifications, hence we omit most of the details.

From the  Euler sequence on ${\mathbb P}^{n}$ restricted to $X$
$$0\lra {\mathcal O}_{X} \lra {\mathcal O}_{X}(1)^{\oplus (n+1)} \lra T_{{{\mathbb P}^{n}}|{X}} \lra  0$$and the facts that
$H^{i}(X,{\mathcal O}_{X})= H^{i}(\FF_e,{\mathcal O}_{\FF_e})= 0, \;\text{for}\;i\ge 1,$ and $X$ is non--degenerate in $\mathbb P^n$, with $n$ as in \eqref{eq:ndnew},  one has:

\begin{equation}\label{eq:tangcom}
h^0(X,T_{{ {\mathbb P}^{n}}|{X}}) = (n+1)^2 -1 \; \; \mbox{and} \;\;
h^i(X,T_{{ {\mathbb P}^{n_e}}|{X}})=0, \; \mbox{for} \; i\ge 1.
\end{equation}

The normal sequence
\begin{eqnarray}\label{eq:seqnorm} 0\lra T_{X} \lra T_{{ {\mathbb P}^{n}}|{X}}
\lra N \lra 0
\end{eqnarray}
 therefore gives
\begin{eqnarray}\label{eq:cohnorm}
H^{i}(X,N) \cong H^{i+1}(X,T_{X}) \qquad {\text {for} \quad  i\ge 1.}
\end{eqnarray}

\begin{claim}\label{cl:com} $H^i(X,N) = 0$, for $i \geq 1$.
\end{claim}
\begin{proof}[Proof of Claim \ref{cl:com}] \, From  \eqref{eq:tangcom} and \eqref{eq:seqnorm}, one has $h^j(X,N) = 0$, for $j\ge 3$.
For the other cohomology spaces, we use \eqref{eq:cohnorm}. In order to compute $H^{j}(X,T_{X})$ we use the scroll map
$\varphi:{\mathbb P}({\mathcal E})\lra \FF_e$ and we consider the relative cotangent bundle sequence:
\begin{eqnarray}\label{eq:relativctgbdl}
0\to \varphi^{*}({\Omega}^1_{\FF_e})\to {\Omega}^1_{X}
\to {\Omega}^1_{X|{\FF_e}} \lra 0.
\end{eqnarray}The adjunction theoretic characterization of the scroll and \eqref{eq:C1C2} give
$$K_{X}= -2L+\varphi^{*}(K_{\FF_e}+c_1({\mathcal E}))=-2L+\varphi^{*}(K_{\FF_e}+4C_{0}+ (3e+6+t+b) f)$$
thus
$${\Omega}^1_{X|{\FF_e}}=K_{X}+\varphi^{*}(-K_{\FF_e})=
-2L+\varphi^{*}(4C_{0}+(3e+6+t+b) f)$$
which, combined with the dual of \brref{eq:relativctgbdl}, gives
\begin{eqnarray}\label{eq:relativetgbdl}
0 \to 2L-\varphi^{*}(4C_{0}+(3e+6+t+b) f) \to T_{X}
\to \varphi^{*}(T_{\FF_e}) \to 0.
\end{eqnarray}

As in \cite[Theorem 4.4]{fa-fla}, if $e\neq 0,$
\begin{eqnarray}\label{eq:50}
h^0(X, \varphi^*(T_{\FF_e}))= h^0(\FF_e, T_{\FF_e}) = e+5, \nonumber\\
h^1(X, \varphi^*(T_{\FF_e}))= h^1(\FF_e, T_{\FF_e}) = e-1,\\
h^j(X, \varphi^*(T_{\FF_e}))= h^j(\FF_e, T_{\FF_e}) = 0, \; \mbox{for} \; j \ge 2, \nonumber
\end{eqnarray}whereas
\begin{eqnarray}\label{eq:50bis}
h^0(X, \varphi^*(T_{\FF_0}))= h^0(\FF_0, T_{\FF_0}) = 6,\\
h^j(X, \varphi^*(T_{\FF_0}))= h^j(\FF_0, T_{\FF_0}) = 0, \; \mbox{for} \; j \ge 1. \nonumber
\end{eqnarray}
\medskip
\noindent
The cohomology of $2L-\varphi^{*}(4C_{0}+(3e+6+t+b) f)$ in \eqref{eq:relativetgbdl} is computed as in \cite[Theorem 4.4]{fa-fla}.
Since  $R^{i}\varphi_{*}(2L)=0$ for $i \ge 1$ (see \cite[Ex. 8.4, p. 253]{H}), projection formula and Leray's isomorphism give
$$H^i(X, 2L-\varphi^{*}(4C_{0}+(3e+6+t+b) f)) \cong \\
H^i({\FF_e}, Sym^2{\mathcal E}\otimes (-4C_0-(3e+6+t+b) f)),
\; \forall \; i \ge 0.$$


As in the proof of \cite[Theorem 4.4]{fa-fla}, the fact that $\mathcal E$ fits in \eqref{eq:al-beca} implies there exists a finite filtration
$$Sym^2({\mathcal E})=F^{0}\supseteq F^{1}\supseteq F^{2} \supseteq F^{3}=0$$s.t.
$$ F^{0}/F^{1}\cong 2B, \; F^{1}/F^{2} \cong A+ B, \; F^{2} \cong 2A,$$since $F^3=0$ (for technical details, we refer the reder to  \cite[Theorem 4.4]{fa-fla}).
Thus, we get the following exact sequences
\begin{equation}\label{filtraz}
0 \to F^{1} \to Sym^2({\mathcal E}) \to 2B\to 0 \;\;\; \mbox{and} \;\;\; 0 \to  2A \to F^{1}  \to A+B\to 0,
\end{equation}Tensoring the two exact sequences in \brref{filtraz} by $-c_1({\mathcal E})=-4C_{0}-(3e+6+t+b) f=-A-B$, we get respectively:
\begin{equation}\label{filtraz1tw}
0 \to F^{1} (-4C_{0}-(3e+6+t+b) f) \to Sym^2({\mathcal E}) \otimes (-4C_{0}-(3e+6+t+b) f) \to B-A\to 0,
\end{equation}
\begin{equation}\label{filtraz2tw}
0 \to A-B\to F^{1} (-4C_{0}-(3e+6+t+b) f)  \to {\mathcal O}_{F_{e}}\to 0.
\end{equation}

From \brref{eq:al-beca3} it follows that $A-B=2C_0+(3e+4+t-b)f$. Now  $R^{0}{\pi}_*(2C_0 + (3e+4+t-b)f) \cong (\Oc_{\Pp^1} \oplus \Oc_{\Pp^1} (-e) \oplus \Oc_{\Pp^1} (-2e)) \otimes \Oc_{\Pp^1} (3e+4+t-b)$ and $R^{i}{\pi_e}_*(2C_0 + (3e+4+t-b)f)=0$, for $i>0$.  Hence, from Leray's isomorphism and Serre duality, we have
\begin{eqnarray*}
h^j(A - B)&=& h^j(\Pp^1, (\Oc_{\Pp^1} \oplus \Oc_{\Pp^1} (-e) \oplus \Oc_{\Pp^1} (-2e)) \otimes \Oc_{\Pp^1} (3e+4+t-b))\\
&=& h^j(\Oc_{\Pp^1} (3e+4+t-b)) + h^j(\Oc_{\Pp^1} (2e+4+t-b))+ h^j(\Oc_{\Pp^1} (e+4+t-b))\\
&=& 0, \text{ if } j\ge 2;
\end{eqnarray*}
\begin{eqnarray*}
h^1(A - B)&=& h^1(\Oc_{\Pp^1} (3e+4+t-b)) + h^1(\Oc_{\Pp^1} (2e+4+t-b))+ h^1(\Oc_{\Pp^1} (e+4+t-b))\\
&\cong&h^0(\Oc_{\Pp^1} (b-6-t-3e)) + h^0(\Oc_{\Pp^1} (b-6-t-2e))+ h^0(\Oc_{\Pp^1} (b-6-t-e)).
\end{eqnarray*}This implies that, if $b< 6+t+e$, then $h^1(A - B)=0$.

Since by \brref{boundbl} we have $b=b_{\ell} < 2e+4+t$, notice that
$2e+4+t\le e+6+t$ is equivalent to our numerical assumption $e\le 2$.
In particular, under the assumptions \eqref{boundbl} and $e \leq 2$,
the case $b\ge 6+t+e$ cannot occur. Hence, under these assumptions,
no indecomposable vector bundles can arise, i.e. any $\mathcal E$ is such that
\begin{equation}\label{eq:split}
{\mathcal E}=A\oplus B.
\end{equation} Moreover the condition $h^1(A-B)=0$,  along with  $h^{i}({\mathcal O}_{\FF_{e}})=0,$ for $i\ge1,$ and  the cohomology associated to \brref{filtraz2tw}
give $h^i(F^{1} (-4C_{0}-(3e+6+t+b) f)=0$ for $i\ge 1$.

We now compute the cohomology of $B-A$. Using Serre duality,
\begin{eqnarray*}
H^j( B-A)&=& H^j(-2C_0-(3e+4+t-b)f)\\
&\cong& H^{2-j}(-2C_0-(e+2)f+2C_0+(3e+4+t-b)f))\\
&\cong&H^{2-j}((2e+2+t-b)f)\\
&\cong& H^{2-j}(\Oc_{\Pp^1} (2e+2+t-b)).
\end{eqnarray*} Notice that, when  $b\ge 2e+3+t$ one has $h^2(B-A)=0$.

Thus numerical assumption $b = 2e+3+t$ in the statement is compatible with \eqref{boundbl}
and ensures the vanishing of $H^2(B-A)$. It also implies $H^1( B-A)\cong H^1(\Oc_{\Pp^1}(-1)) = 0$.

From the cohomology sequence associated to \brref{filtraz1tw} it follows that  if $b=2e+3+t$ then
\begin{xalignat}{1}
\notag
h^j(X, 2L-\varphi^{*}(4C_{0}+(3e+6+t+b) f)) & = h^j(2L-\varphi^{*}(4C_{0}+(5e+9+2t) f))\\ \label{eq:hjsym2E}
 & =
h^j({\FF_e}, Sym^2{\mathcal E}\otimes (-4C_0-5e-9-2
t) f) \\ \notag
& = 0, \text{ for } j \ge 1.
\end{xalignat}

Using \eqref{eq:50}, \eqref{eq:50bis} and \eqref{eq:hjsym2E} in the cohomology
sequence associated to \brref{eq:relativetgbdl}, we get
\begin{equation}\label{eq:cohTXe}
h^j(X, T_{X}) = 0, \; \mbox{for} \; j \ge 2.
\end{equation}
Moreover
\begin{eqnarray}\label{eq:coh1TXe}
h^1(X, T_{X}) = \left\{\aligned  e-1&\qquad\;  \mbox{for} \; e\neq 0\\
 0& \qquad\; \mbox{for} \; e=0.\\
  \endaligned\right.
\end{eqnarray}
Isomorphism \brref{eq:cohnorm} concludes the proof of Claim \ref{cl:com}.
\end{proof}

Claim \ref{cl:com}, together with the fact that smoothness is an open condition, implies that
there exists an irreducible component $\mathcal X_e$ of
$\mathcal H_3^{d,n}$ which is generically smooth, of the expected dimension $\dim(\mathcal X_e) = h^0(X,N)= \chi(N)$, such that $[X]$ lies in its smooth locus
(recall \eqref{eq:tangspace},\,\eqref{eq:expdim}).


The Hirzebruch-Riemann-Roch theorem gives
\begin{eqnarray}\label{chiN}
\chi(N) &=& \frac{1}{6}(n_1^3-3n_1n_2+3n_3)+
\frac{1}{4}c_1(n_1^2-2n_2) \\
& & + \frac{1}{12}(c_1^2+c_2)n_1 +(n-3)\chi({\mathcal O}_{X}),\nonumber
\end{eqnarray}
where $n_i:=c_i(N)$ and $ c_i:=c_i(X).$

Setting, for simplicity,  $K:= K_{X}$,  Chern classes of $N$  can be obtained from \brref{eq:seqnorm}:
\begin{eqnarray}\label{valueniscrollP}
 n_1&=&K+(n+1)L;  \nonumber \\
\;\;\;\;\;\;\;\; n_2&=&\frac{1}{2}n(n+1)L^2+(n+1)LK+K^2-c_2; \\
n_3&=&\frac{1}{6}(n-1)n(n+1)L^3+\frac{1}{2}n(n+1)KL^2+ (n+1)K^2L \nonumber\\
& &-(n+1)c_2L-2c_2K+K^3-c_3. \nonumber
\end{eqnarray} The numerical invariants of $X$ can be easily computed by:
\begin{xalignat}{2}
  KL^2 &= -2d+6e+28+6t+6b;&  K^2L &= 4d-20b-20t-20e-96;\notag \\
  c_2L &= 2e+24+2b+2t;& K^3 &= -8d+48b+48t+48e+240;\notag \\
  -Kc_2 &= 24;&  c_3 &= 8. \nonumber
\end{xalignat}

Plugging these in \brref{valueniscrollP} and then in \brref{chiN}, one gets
$$\chi(N)=(d-3e-3b-3t-12)n+122+21t+21e+21b-3d.$$From \eqref{eq:ndnew}, one has
$d = 8e+5b+7t+40$ and $n = 5e+2b+4t+27 $; in particular
$$d-3e-3b-3t-12= n+1.$$ Thus
$$\chi(N) = (n+1) n  + 2 -3e +6b = (n+1) n + 9e + 20 + 6t,$$as in \eqref{eq:expdim},
with $n = 9e + 33 + 6t$  since $ b = 2e + 3 + t$.
\end{proof}

\begin{rem}\label{rem:flaminio} \normalfont{The proof of Theorem \ref{thm:Hilbertscheme} gives
\begin{equation}\label{eq:comNe}
h^0(N) = (n+1) n + 9e + 20 + 6t, \;\; h^i(N) = 0, \; i\geq 1.
\end{equation}
Using  \eqref{eq:tangcom} and \eqref{eq:comNe} in
the exact sequence \eqref{eq:seqnorm} and the values of $b$, $n$ and $d$ as in Theorem \ref{thm:Hilbertscheme},  one gets
\begin{equation}\label{eq:5.6}
\chi(T_{X}) = n - 6b + 3 e - 2 = 8e - 4b + 4t+25 = 13.
\end{equation}Moreover, from \eqref{eq:seqnorm} and \eqref{eq:tangcom}, one has:
\begin{equation}\label{eq:5.flaminio}
0 \to H^0(T_{X}) \to H^0(T_{\Pp^{n}|_{X}}) \stackrel{\alpha}{\to} H^0(N) \stackrel{\beta}{\to} H^1(T_{X}) \to 0.
\end{equation}
}
\end{rem}

\begin{cor}\label{flaminio} With the same assumptions as in Theorem \ref{thm:Hilbertscheme} one has:

\begin{itemize}
\item[i)] if $e\neq 0$,
$$h^0(T_{X}) =e+12, \;\; h^1(T_{X}) = e-1, \;\; h^j(T_{X}) = 0, \; \mbox{for} \; j \ge 2;$$
\item[ii)] if $e=0,$
$$h^0(T_{X}) =13, \;\; h^j(T_{X}) = 0, \; \mbox{for} \; j \ge 1.$$
\end{itemize}
\end{cor}
\begin{proof} Let $e\neq 0$, then $h^j(T_{X}) = 0$, for $j \ge 2$, is \eqref{eq:cohTXe} and $h^1(T_{X}) = e-1$, from \eqref{eq:coh1TXe}. We now use  \eqref{eq:5.6} to get that $h^0(T_{X}) =9e-4b+4t+20=e+12.$ A similar argument gives the desired values for $h^0$ and $h^j$ in the case $e=0.$
\end{proof}

The next result shows that, for $e=2$, scrolls arising from Proposition \ref{prop:Xscroll} do not fill up the component
$\mathcal X_2 \subseteq \mathcal H_3^{d,n}$.

\begin{theo}\label{thm:parcount}  Assumptions as in Theorem \ref{thm:Hilbertscheme}.
Let $\mathcal Y_e$ be the locus in $\mathcal X_e$ filled-up by $3$-fold scrolls $X$ as in Proposition
\ref{prop:Xscroll}. Then
\begin{eqnarray*}
{\rm codim}_{\mathcal X_e} ( \mathcal Y_e) = \left\{\aligned  e-1&\qquad\;  \mbox{for} \; e\neq 0\\
 0& \qquad\; \mbox{for} \; e=0\\
  \endaligned\right.
\end{eqnarray*}
 \end{theo}
\begin{proof} If $\tau$ denotes the number of parameters counting isomorphism classes of projective bundles $\Pp({\mathcal E})$ as in Proposition \ref{prop:Xscroll}, then $\tau = 0$  since $\mathcal E = A \oplus B$ (see \eqref{eq:split}).  Therefore
$X \cong \Pp(A \oplus B)$ is uniquely determined by $A$ and $B$. Thus, by construction,
$\dim(\mathcal Y_e) = \dim({\rm Im}(\alpha))$. From \eqref{eq:5.flaminio},  $\dim({\rm Coker} (\alpha)) = h^1(T_{X})$ so

\begin{eqnarray*}
{\rm codim}_{\mathcal X_e} ( \mathcal Y_e) = \dim({\rm Coker}(\alpha))= \left\{\aligned  e-1&\qquad\;  \mbox{for} \; e\neq 0\\
 0& \qquad\; \mbox{for} \; e=0\\
  \endaligned\right.,
\end{eqnarray*}
\end{proof}

\section {Open Questions} We have seen in Theorem \ref{thm:parcount}  that  the locus  $\mathcal Y_e$  in $\mathcal X_e$ of  $3$-fold scrolls
$X$ as in Proposition \ref{prop:Xscroll},  does not necessarily  fill up $\mathcal X_e$ if $e=2$.  The following problems arise naturally:

\begin{enumerate}
\item[1)] Describe a variety $Z$ which is a candidate to represent the general point of the component $\mathcal X_2;$

\item[2)] Assuming that a description of a variety $Z$ as in 1) above is achieved, interpret the projective degeneration
of $Z$ to $X$, where $[X] \in \mathcal Y_2$, in terms of vector bundles on Hirzebruch surfaces;

\item[3)] Extend the results in this note to the case $e \ge3$.
\end{enumerate}

\end{document}